\documentclass[11pt,a4paper]{amsart}
\usepackage{amsfonts}
\usepackage{amsthm}
\usepackage{amsmath}
\usepackage{amssymb}
\usepackage{amscd}
\usepackage{t1enc}
\usepackage[margin=2.9cm]{geometry}
\usepackage{epstopdf} 
\usepackage{times}
\usepackage{enumerate}
\usepackage{changepage}
\usepackage{mfirstuc}  
\usepackage{tikz}
\usepackage{cite}
\usepackage{booktabs}
\usepackage{tabularx}
\numberwithin{equation}{section}

\theoremstyle{plain}
\newtheorem{Th}{Theorem}[section]
\newtheorem{Lemma}[Th]{Lemma}

\newtheorem{Cor}[Th]{Corollary}
\newtheorem{Prop}[Th]{Proposition}
\newtheorem{Def}[Th]{Definition}

\theoremstyle{definition}

\theoremstyle{remark}
\newtheorem*{remark}{Remark}

\newcommand{\Q}{\mathbb{Q}}
\newcommand{\R}{\mathbb{R}}
\newcommand{\Z}{\mathbb{Z}}
\newcommand{\F}{\mathbb{F}}

\newcommand{\LF}{\mathrm{L}}

\newcommand{\ord}{\operatorname{ord}}
\newcommand{\Vol}{\operatorname{Vol}}
\newcommand{\Tr}{\operatorname{Tr}}
\newcommand{\N}{\operatorname{N}}

\begin{document}

\title{On inverted Kloosterman sums over finite fields}
\author{Xin Lin}
\address{Department of Mathematics, Shanghai Maritime University, Shanghai 201306, PR China.}
\author{Daqing Wan}
\email{xlin1126@hotmail.com}
\address{Department of Mathematics, University of California, Irvine, CA
  92697-3875 USA.}
 \email{dwan@math.uci.edu} 
  
  \subjclass[2020]{11T23, 11L05, 11L07, 11S40}

\keywords{Inverted Kloosterman sums, Exponential sums, L-function, Finite field}

\begin{abstract} The classical $n$-variable Kloosterman sums over finite fields are well understood by 
Deligne's theorem from complex point of view and by Sperber's theorem from $p$-adic point of view. In this paper, we study the 
complex and $p$-adic estimates of {\it inverted} $n$-variable Kloosterman sums,  addressing a question of N. Katz (1995). 
We shall give two complex estimates. The first one is elementary based on Gauss sums. The second estimate is deeper, 
depending on the cohomological results  of Adolphson-Sperber, Denef-Loeser 
and Fu for twisted toric exponential sums. This deeper result assumes that the characteristic $p$ does not divide $n+1$. 
Combining with Dwork's $p$-adic theory, we also determine the exact $p$-adic valuations for zeros and poles of the L-function associated to {\it inverted} $n$-variable Kloosterman sums in the case $p \equiv 1 \mod (n+1)$. 
As we shall see, the {\it inverted} $n$-variable Kloosterman sum is more complicated than the classical $n$-variable Kloosterman sum in all aspects in the sense that our understanding is less complete, partly because the Hodge numbers are now mostly $2$ 
instead of $1$.

\end{abstract}

\maketitle

\allowdisplaybreaks[4]
\section{Introduction}\label{sec1}

Let $\F_{q}$ be the finite field of $q$ elements with characteristic $p$. 
Let $\psi: \F_{q} \rightarrow \mathbb{C}^*$ be a nontrivial additive character and let
$\chi_1,\ldots,\chi_{m}: \F_{q}^* \rightarrow \mathbb{C^*}$ be multiplicative characters.
A classical problem in number theory is to give a good estimate for the mixed character sum
\begin{align}\label{equ0}
	\sum_{x_i\in\F^*_q}\chi_1(f_1)\cdots\chi_m(f_m)\psi(f),
\end{align}
where  $f,f_1,\ldots,f_m\in\F_{q}\left[x_1^{\pm1}, \ldots, x_n^{\pm1} \right]$ are Laurent polynomials.
Reducing $m$ if necessary, we may assume that all the $\chi_i$'s are non-trivial. Using the Gauss sum
$$G(\chi)=\sum_{x\in\F_q^*}\chi(x)\psi(x),$$ 
one obtains the well known relation 
\begin{align*}
	&\sum_{x_i,y_j\in\F^*_q}\overline{\chi_1}(y_1)\cdots\overline{\chi_m}(y_m)\psi\left(f+y_1f_1+\cdots+y_mf_m\right)\\
	&= G(\overline{\chi_1})\cdots G(\overline{\chi_m})\sum_{x_i\in\F^*_q}\chi_1(f_1)\cdots\chi_m(f_m)\psi(f).
\end{align*}
As the Gauss sums are well-understood, the study of  (\ref{equ0}) is reduced to the study of the following 
type of twisted toric exponential sum 
\begin{align}\label{equ01}
	\sum_{x_i\in\F^*_q}\overline{\chi_1}(x_1)\cdots\overline{\chi_n}(x_n)\psi\left(f(x_1,\cdots,x_n)\right),
\end{align}
where some of the $\chi_i$'s may be trivial. 
This type of twisted toric exponential sum has been studied extensively in the literature, most notably by Adolphson-Sperber\cite{AS1987, AS1989,AS1990, AS1993} via Dwork's $p$-adic cohomology, and by Denef-Loeser\cite{DL1991} and Fu\cite{Fu2009,Fu2016} via Grothendieck's $\ell$-adic cohomology.  
For the complex estimate, both approaches depend on Deligne's theorem on the Weil conjectures. A sharp estimate is obtained when $f$ is non-degenerate 
with respect to its Newton polyhedron $\Delta(f)$.  For arbitrary $f$, the sum is still far from well understood. 

An important example of toric exponential sums is the classical $n$-variable Kloosterman sum, where $\chi_1=\cdots = \chi_n=1$ and 
$$f(x_1,\cdots, x_n) = x_1 +\cdots + x_n + \frac{b}{x_1\cdots x_n},  ~ b \in \F_q^*.$$
In this case, the complex weights were determined by Deligne's well known theorem,  and the $p$-adic slopes were determined by Sperber's theorem \cite{Spe1980}. 
It should be noted that for twisted $n$-variable Kloosterman sum (when some of the $\chi_i$'s are non-trivial),  
the $p$-adic slopes are not completely determined in general, except in the case $n=1$ for which Adolphson-Sperber\cite{AS1987c} obtained the generic $p$-adic slopes.  If we invert the above Laurent polynomial and consider 
the following rational function 
$$f(x_1,\cdots, x_n) =\frac{1}{ x_1 +\cdots + x_n + \frac{b}{x_1\cdots x_n}},  ~ b \in \F_q^*,$$
which is no longer a Laurent polynomial, we are led to the so-called {\it inverted} Kloosterman sum. The study of 
such sums goes back to N. Katz\cite{Kat1995}. 

More precisely, in this paper, we study the following \emph{twisted inverted $n$-variable Kloosterman sum} defined by
\begin{align*}
	S_n(\chi,b)&=\mathop{\sum_{x_1\cdots x_{n+1}=b, ~x_i\in\F_q^*}}_{x_1+\cdots+x_{n+1}\neq 0} \chi_1(x_1)\cdots\chi_{n+1}(x_{n+1})\psi\left(\frac{1}{x_1+\cdots+x_{n+1}}\right)\\
	&=\mathop{\sum_{x_1+\cdots+x_{n}+\frac{b}{x_1\cdots x_{n}}\neq 0}}_{x_i\in\F_q^*} \chi_1(x_1)\cdots\chi_{n}(x_n)\chi_{n+1}\left(\frac{b}{x_1\cdots x_{n}}\right)
	\psi\left(\frac{1}{x_1+\cdots+x_{n}+\frac{b}{x_1\cdots x_{n}}}\right),
\end{align*}
where $b\in \F^*_{q}$ and $n\geq 1$.  When $n=1$, Katz\cite{Kat1995} obtained a sharp upper bound for $S_1(\chi,b)$. This result along with the papers of Angel\cite{Ang1996} and Evans\cite{Eva1995} proves that finite upper half plane graphs are Ramanujan in characteristic 2. 
In \cite{Kat1995}, Katz raised the question: what can be said for $S_n(\chi,b)$ when $n\geq 1$? 
The aim of this paper is to study these inverted $n$-variable 
Kloosterman sums from both complex and $p$-adic point of views. As we shall see, this 
class of sums is very interesting as various new features and additional difficulties arise. 

\begin{remark}
The exact sum introduced in \cite{Kat1995} is the following related sum, 
$$T_n(\chi,b)=\mathop{\sum_{x_1\cdots x_{n+1}=1}}_{x_1+\cdots+x_{n+1}\neq 0} \chi_1(x_1)\cdots\chi_{n+1}(x_{n+1})\psi\left(\frac{b}{x_1+\cdots+x_{n+1}}\right).$$
Upon the change of variables $x_i \rightarrow bx_i$, one sees that 
$$T_n(\chi,b) = S_n(\chi,\frac{1}{b^{n+1}}) \chi_1\cdots \chi_{n+1}(b).$$
Thus, the two families of sums $S_n(\chi,b)$ and $T_n(\chi, b)$ are essentially equivalent.  
We work with $S_n(\chi,b)$ as it is the closer inverted analogue of the classical Kloosterman sum. 

\end{remark}

For complex estimate of $S_n(\chi, b)$, the best one can hope for would be a  square root cancellation in the sum $S_n(\chi,b)$, i.e. 
\begin{align*}
	S_n(\chi,b)=O_n(q^{\frac{n}{2}}).
\end{align*}
As we shall see, this is not true for $n>2$ when  $\chi_1=\cdots=\chi_{n+1}$, in which case, $S_n(\chi,b)$ has the main term 
$-q^{n-1}\chi_1(b)$ whose exponent $n-1$ is larger than the exponent $n/2$.  

We shall give two different estimates for $S_n(\chi,b)$. 
The first estimate is based on an elementary method via Gauss sums which already shows the new feature of 
a non-trivial main term when  $\chi_1=\cdots=\chi_{n+1}$. We obtain the following simple estimate for $S_n(\chi,b)$.
\begin{Th}\label{thm0} 
	Notations as above. If $\chi_1=\cdots=\chi_{n+1}$, we have
	\begin{align*}
		\left|S_n(\chi,b)+\frac{(q-1)^{n}}{q}\chi_1(b) \right| \leq q^{\frac{n+1}{2}}.
	\end{align*}
	If $\chi_i\neq\chi_j$ for some $i\neq j$, we have
	\begin{align*}
		\left|S_n(\chi,b)\right| \leq q^{\frac{n+1}{2}}.
	\end{align*}
\end{Th}
For large $q$, the error term $q^{(n+1)/2}$ is not the optimal square root cancellation yet. To obtain the deeper square root 
cancellation with error term $O_n(q^{n/2})$, we 
reduce $S_n(\chi,b)$ to a certain twisted toric exponential sum $S^*_{k}(\chi,f)$ which can be handled by the results of Adolphson-Sperber, Denef-Loeser and Fu. 
We check that the related Laurent polynomial $f$ is non-degenerate 
if $p$ does not divide $n+1$. This gives our second estimate.

\begin{Th}\label{thm2} 
	Notations as above. Suppose $p\nmid (n+1)$. If $\chi_1=\cdots=\chi_{n+1}$, we have
	\begin{align*}
		|S_n(\chi,b)+\frac{(q-1)^{n}-(-1)^n}{q}\chi_1(b) |\leq (2n+1)q^{\frac{n}{2}}.
	\end{align*}
	If $\chi_i\neq\chi_j$ for some $i\neq j$, we have
	\begin{align*}
		|S_n(\chi,b)|\leq 2(n+1)q^{\frac{n}{2}}.
	\end{align*}
\end{Th}

It is clear that the estimate in Theorem \ref{thm2} is better than the estimate in Theorem \ref{thm0} if  $q> 4(n+1)^2$. 
It is expected that Theorem \ref{thm2} (with possibly a better constant) remains true when $p$ divides $n+1$. But in this singular case, the above toric sum results do not apply and one would need a different approach.

For $p$-adic slopes, we focus on the simpler untwisted case. When $\chi_1=\cdots=\chi_{n+1}$, we study the $p$-adic valuations for the reciprocal roots and poles of the generating L-function of the inverted $n$-variable Kloosterman sum. We show that 
the Newton polygon agrees with the Hodge polygon if and only if $p \equiv 1 \mod n+1$. As a consequence, this completely 
determines the $p$-adic slope sequence when $p \equiv 1 \mod n+1$. This is described more precisely below.

Our approach is to reduce the generating L-function to a certain untwisted toric L-function. To construct the relationship between the two L-functions, we need to consider the inverted Kloosterman sum defined over every finite extension $\F_{q^k}$. 
Suppose $\chi_1=\cdots=\chi_{n+1}$,
the inverted $n$-variable Kloosterman sum over $\F_{q^k}$ is defined by 
\begin{align*}
	S_{k,n}(b)=
	\mathop{\sum_{x_1+\cdots+x_{n}+\frac{b}{x_1\cdots x_{n}}\neq 0}}_{x_i\in\F_{q^k}^*} 
	\psi\left(\Tr_k\left(\frac{1}{x_1+\cdots+x_{n}+\frac{b}{x_1\cdots x_{n}}}\right)\right),
\end{align*}
where $b\in \F^*_{q}$ and $\Tr_k: \F_{q^k} \rightarrow \F_{q}$ is the trace map.
The generating L-function of $S_{k,n}(b)$ is defined by 
\begin{align*}
	\LF_{n}(b,T)=\exp \left(\sum^\infty _ {k=1} S_{k,n}(b) \frac{T^k}{k} \right).
\end{align*}
 Applying some systematic results available for the related toric L-function, we obtain the complex and $p$-adic absolute values for all the reciprocal roots and poles of $\LF_{n}(b,T)$ under given restrictions on $p$.

\begin{Th}\label{thm1}
	Suppose $p\nmid (n+1)$. The L-function is a rational function of the following form:  
	\begin{align*}
		\LF_{n}(b,T)^{(-1)^{n+1}}=(1-T)^{(n+1)}
		\prod_{j=2}^{n}\left(1-q^{j-1}T\right)^{\binom{n}{j}(-1)^{j-1}}	
		\prod^{2n}_{i=1}(1-\alpha_iT).
	\end{align*}
	As complex numbers, the reciprocal roots $\alpha_i$ satisfy $|\alpha_i|= q^{\frac{n}{2}}$ for all $1\leq i\leq 2n$. 
	
	If $p\equiv 1\bmod (n+1)$, viewing the $\alpha_i$'s as $p$-adic numbers, the slope sequence $\{ v_q(\alpha_i)\}_{i=1}^{2n}$ in increasing order is given by 
	$$\{0,1,1,2,2,\ldots,n-1,n-1,n\}.$$
\end{Th}

Our results show that for the {\it inverted} $n$-variable Kloosterman sum, if $p$ does not divide $n+1$, the primitive middle cohomology has dimension $2n$, pure of weight $n$ and  with Hodge numbers $\{1, 2, 2, \cdots, 2, 1\}$. This is in contrast 
to the classical $n$-variable Kloosterman sum, where the middle cohomology has dimension $n+1$, pure of weight $n$ 
and with Hodge numbers $\{1, 1, 1, \cdots, 1, 1\}$. The larger Hodge numbers suggest that new features and additional difficulties would likely arise in studying {\it inverted} $n$-variable Kloosterman sums.  

As a corollary of the first part in Theorem \ref{thm1}, we get a slightly better bound for $S_{k,n}(b)$.
\begin{Cor}\label{cor1}
	If $p\nmid (n+1)$, for all integers $k\geq 1$, we have
	\begin{align*}
		|S_{k,n}(b)+\frac{(q^k-1)^n-(-1)^n(q^k+1)}{q^k}|\leq 2nq^{\frac{nk}{2}}.
	\end{align*}
\end{Cor}
When $k=1$, the exponential sum $S_{k,n}(b)$ reduces to $S_n(\chi,b)$ with $\chi_1=\cdots=\chi_{n+1}$. 
Explicitly, the estimate in Corollary \ref{cor1} is better than the estimate in the  first case of Theorem \ref{thm2}. 
We remark that the condition $p\equiv 1\bmod (n+1)$ in the second part of Theorem \ref{thm1} is necessary and 
sufficient for the same conclusion to hold. Thus, if $p\not\equiv 1\bmod (n+1)$, the slope sequence will be 
strictly different, but we do not know the exact slope sequence in this case. 


The rest of this paper is organized as follows. In section \ref{sec2}, we review some technical methods including Adolphson-Sperber's theorems and Dwork's theory on toric exponentials sums. In section \ref{sec3}, we use these methods to prove the main results. 

We end the 
introduction by mentioning several recent references \cite{FW2021}\cite{Li2021}\cite{YZ2022}\cite{CL2022}\cite{LC2022} which applied some of the related toric techniques in treating different classes of {\it non-toric} $n$ variable exponential sums arising from analytic number theory.

\section{Preliminaries on toric exponential sums}\label{sec2}
\subsection{Rationality of the toric L-function}\label{subsec21}
Let $f\in\F_{q}\left[x_1^{\pm1}, \ldots, x_n^{\pm1} \right]$ be a Laurent polynomial and its associated twisted toric exponential sum is defined to be 
\begin{align}\label{equ1}
	S^*_k(\chi,f)=\sum_{x_i \in \F^*_{q^k}} \chi_1(\N_k(x_1))\cdots\chi_n(\N_k(x_n))\psi(\Tr_k(f)),
\end{align}
where $\Tr_k: \F_{q^k} \rightarrow \F_{q}$ is the trace map, $\N_k: \F_{q^k} \rightarrow \F_{q}$ is the norm map, $\chi_1,\ldots,\chi_n: \F_{q}^* \rightarrow \mathbb{C^*}$ are multiplicative characters and $\psi: \F_{q} \rightarrow \mathbb{C}^*$ is a nontrivial additive character. 
A classical problem in number theory is to estimate the absolute values of $S^*_k(\chi,f)$. 

A well known theorem of Dwork-Bombieri-Grothendieck\cite{Dwo1960,Bom1966a,Gro1968} says that the generating L-function of $S^*_k(\chi,f)$ is a rational function:
\begin{equation*}
	\LF^*(\chi,f,T)=\exp  \left(\sum^\infty _ {k=1} S^*_k(\chi,f) \frac{T^k}{k} \right)
	=\frac{\prod^{d_1}_{i=1}(1-\alpha_iT)}{\prod^{d_2}_{j=1}(1-\beta_jT)},
\end{equation*}
where all the reciprocal zeros and poles are non-zero algebraic integers. 
In particular, when all of $\chi_i$ are trivial characters, the untwisted toric exponential sum and the associated L-function are denoted by $S^*_k(f)$ and $\LF^*(f,T)$ respectively.

Through logarithmic derivatives, we have
\begin{align}\label{expo}
	S^*_k(\chi,f)=\sum^{d_2}_{j=1}\beta_j^k-\sum^{d_1}_{i=1}\alpha_i^k, \quad k \in \Z_{\geq1}.
\end{align}
Thus, the estimate of $S^*_k(\chi,f)$ is reduced to understanding all the absolute values of the reciprocal zeros $\alpha_i$ and poles $\beta_j$.
Deligne's theorem on Riemann hypothesis \cite{Del1980} describes the bounds for all the absolute values of $\alpha_i$ and $\beta_j$ in general. The complex absolute values of reciprocal zeros and poles 
satisfy 
\begin{equation*}
	|\alpha_i|=q^{u_i/2}, |\beta_j|=q^{v_j/2}, u_i\in \Z \cap [0,2n], v_j\in \Z \cap [0,2n].
\end{equation*}
For non-archimedean absolute values, Deligne proved that $|\alpha_i|_\ell=|\beta_j|_\ell=1$ when $\ell$ is a prime and $\ell\neq p$. For $p$-adic absolute values, one has
\begin{equation*}
	|\alpha_i|_p=q^{-r_i}, |\beta_j|_p=q^{-s_j}, r_i\in \Q \cap [0,n], s_j\in \Q \cap [0,n].
\end{equation*}
The integer $u_i$ (resp. $v_j$) is called the \textit{weight} of $\alpha_i$ (resp. $\beta_j$) and the rational number $r_i$
(resp. $s_j$) is called the \textit{slope} of $\alpha_i$ (resp. $\beta_j$).

In the past few decades, there has been tremendous interest in determining the weights and slopes of the generating L-functions. Without any further condition on the Laurent polynomial $f$, it is even hard to determine the number of reciprocal roots and poles. Most of the existing work about the weights and slopes relies on a suitable smoothness condition. For toric exponential sums, 
this usually means the non-degenerate condition, see below for the precise definition.

Let 
\begin{equation}\label{equ11}
	f(x_1, \ldots x_n)=\sum^J_{j=1} a_j x^{V_j}
\end{equation}
be a Laurent polynomial with $a_j \in \F ^*_{q}$ and $V_j=(v_{1j},\ldots, v_{nj})\in \Z^n$ $(1\leq j\leq J)$. The \emph{Newton polyhedron} of $f$, $\Delta(f)$, is defined to be the convex closure in $\R^n$ generated by the origin and the lattice points $V_j$ ($1\leq j \leq J$). For $\delta \subset \Delta(f)$, let the Laurent polynomial
\begin{equation*}
	f^{\delta}=\sum_{V_j\in \delta}a_j x^{V_j}
\end{equation*}
be the restriction of $f$ to $\delta$. 
\begin{Def}[non-degenerate]\label{def1}
	A Laurent polynomial $f$ is called non-degenerate if for each closed face $\delta$ of $\Delta(f)$ of arbitrary dimension which doesn't contain the origin, the partial derivatives 
	\begin{equation*}
		\left\{ \frac{\partial f^{\delta}}{\partial x_1 }, \ldots, \frac{\partial f^{\delta}}{\partial x_n} \right\}
	\end{equation*} 
	have no common zeros with $x_1\ldots x_n \neq 0$ over the algebraic closure of $\F_q$.
\end{Def}

When $f$ is non-degenerate, Adolphson-Sperber\cite{AS1989} proved that the untwisted toric L-function $\LF^*(f,T)^{(-1)^{n-1}}$ is a polynomial and improved the bound for weight.

\begin{Th}[\cite{AS1989}]\label{th2}
	For any non-degenerate $f\in  \F_{q}[x_1^{\pm1},\ldots, x_n^{\pm1}]$, the associated L-function $\LF^*(f,T)^{(-1)^{n-1}}$ is a polynomial of degree $n!\Vol(\Delta(f))$. Namely, 
	\begin{equation*}
		\LF^*(f,T)^{(-1)^{n-1}}=\prod^{n! \Vol(\Delta(f))}_{i=1}(1-\alpha_i T), \ \alpha_i\not=0. 
	\end{equation*}
	\end{Th}

 For any multiplicative characters $\chi_i$ (nontrivial or trivial), the degree and weights of the twisted toric L-function are studied and completed by Adolphson-Sperber\cite{AS1991,AS1993}, \cite{DL1991}and Fu\cite{Fu2009} under the non-degeneracy assumption. These results lead to the following bound for the twisted toric exponential sum.

\begin{Th}[\cite{DL1991}\cite{AS1993}\cite{Fu2009}]\label{thm3}
	Let $f\in\F_{q}\left[x_1^{\pm1}, \ldots, x_n^{\pm1} \right]$ be a Laurent polynomial with $\Delta=\Delta(f)$. If $f$ is non-degenerate, one has
	\begin{equation*}
		\left|S^*_k(\chi_1,\ldots,\chi_n,f)\right|\leq n!\Vol(\Delta)q^{\frac{nk}{2}}.
	\end{equation*}
\end{Th}

For the slopes of the L-function, the situation is somewhat simpler in the untwisted case, 
otherwise, even the description of Adolphson-Sperber's ``Hodge lower bound" is a little cumbersome.  
Thus, the definitions and theories discussed in the following subsections focus on the untwisted L-function.

\subsection{Newton polygon and Hodge polygon}\label{subsec22}
To determine the $q$-adic slopes of its reciprocal roots, we introduce the $q$-adic Newton polygon.
\begin{Def}[Newton polygon]\label{m1}
	Let $\LF(T)=\sum^n_{i=0}a_i T^i$ $\in 1+T\overline{\Q}_p[T]$, where $\overline{\Q}_p$ is the algebraic closure of $\Q_p$. The $q$-adic Newton polygon of\ $\LF(T)$ is defined to be the lower convex closure of the set of points $\{\left(k,\ord_q(a_k)\right) | k=0, 1,\ldots, n \}$ in $\R^2$.
\end{Def}
\begin{Lemma}[\cite{koblitz2012p}]\label{le4}
	Notations as above. Let $\LF(T)=(1-{\alpha_1}T)\ldots(1-{\alpha_n}T)$ be the factorization of $\LF(T)$ in terms of reciprocal roots $\alpha_i \in \overline{\Q}_p$. Let $\lambda_i=\ord_q\alpha_i$. If $\lambda$ is the slope of the $q$-adic Newton polygon of $L(T)$ with horizontal length $l$, then precisely $l$ of the $\lambda_i$ are equal to $\lambda$.
\end{Lemma}
The $q$-adic Newton polygon of $\LF^*(f,T)^{(-1)^{n-1}}$ is denoted as NP($f$). 
Lemma \ref{le4} relates NP($f$) to the $q$-adic valuation of reciprocal roots of toric L-functions. The definition of NP($f$) relies on the coefficients of L-function, which makes it hard to compute directly. When $f$ is non-degenerate, Adolphson and Sperber proved that $\LF^*(f,T)^{(-1)^{n-1}}$ is a polynomial and NP($f$) has a topological lower bound called Hodge polygon, which is easier to determine. Thus, we shall compute Hodge polygon and consider when the Newton polygon coincides with this lower bound. 

Let $\Delta$ be an $n$-dimensional integral polytope containing the origin in $\R^n$. For $u\in \R^{n}$, the weight function $w(u)$ represents the smallest non-negative real number $c$ such that $u\in c\Delta$. Denote $w(u)=\infty$ if such $c$ doesn't exist. Assume $\delta$ is a co-dimension 1 face of $\Delta$ not containing the origin. Let $D(\delta)$ be the least common multiple of the denominators of the coefficients in the linear equation defining $\delta$, normalized to have constant term 1.
We define the denominator of $\Delta$ to be the least common multiple of all such $D(\delta)$ given by:
\begin{equation*}
	D=D(\Delta)= \mathrm{lcm}_{\delta}D(\delta),
\end{equation*}
where $\delta$ runs over all the co-dimension 1 faces of $\Delta$ that don't contain the origin. It's easy to check 
\begin{equation*}
	w(\Z^n)\subseteq \frac{1}{D(\Delta)}\Z_{\geq0}\cup \{ + {\infty} \}.
\end{equation*} 
For a non-negative integer $k$, let 
\begin{equation}\label{eqW}
	W_{\Delta}(k)=\# \left\{ u \in \Z^n | w(u)= \frac{k}{D} \right\}
\end{equation} 
be the number of lattice points in $\Z^n$ with weight $k/D$. Its generating function is known to be a rational function 
of the following form 
$$\sum_{k=0}^{\infty} W_{\Delta}(k) t^{k/D} = \frac{\sum_{k=0}^{nD} H_{\Delta}(k) t^{k/D}}{(1-t)^n}.$$
This leads to 
\begin{Def}[Hodge number]\label{def5}
	Let $\Delta$ be an $n$-dimensional integral polytope containing the origin in $\R^n$. For a non-negative integer $k$, the $k$-th Hodge number of $\Delta$ is defined to be
	\begin{align}\label{eq3}
		H_{\Delta}(k)=\sum^{n}_{i=0}(-1)^i \binom{n}{i}W_{\Delta}(k-iD).
	\end{align}
\end{Def}
It is known that 
\begin{equation*}
	H_{\Delta}(k)=0, \quad \text{if}\quad k>nD.
\end{equation*}
Based on the Hodge numbers, we define the Hodge polygon of a given polyhedron $\Delta\in \R^n$ as follows.
\begin{Def}[Hodge polygon]\label{def6}
	The Hodge polygon HP($\Delta$) of $\Delta$ is the lower convex polygon in $\R^2$ with vertices (0,0) and 
	\begin{equation*}
		\Q_k=\left( \sum^k_{m=0}H_{\Delta}(m), \frac{1}{D}\sum^k_{m=0}m H_{\Delta}(m) \right), \quad  k=0,1,\ldots, nD,
	\end{equation*}
	where $H_{\Delta}(k)$ is the $k$-th Hodge number of $\Delta$, $k=0,1,\ldots, nD.$
	
	That is, HP($\Delta$) is a polygon starting from origin (0,0) with a slope $k/D$ side of horizontal length $H_{\Delta}(k)$ for $k=0,1,\ldots, nD$. The vertex $\Q_k$ is called a break point if $H_{\Delta}(k+1)\neq 0$ where $k=1,2,\ldots,nD-1$. 
\end{Def}
Note that the horizontal length $H_{\Delta}(k)$ is the number of  lattice points of weight $k/D$ in a certain fundamental domain corresponding to a basis of the $p$-adic cohomology space used to compute the L-function. By a theorem of Adolphson-Sperber, the Hodge polygon is a lower bound of the corresponding Newton polygon. 
\begin{Th}[\cite{AS1989}]\label{thas}
	For every prime p and non-degenerate Laurent polynomial $f$ with $\Delta(f)=\Delta \subset \R^n$, we have 
	\begin{equation*}
		\text{NP}(f) \geq \text{HP}(\Delta),
	\end{equation*}
	where NP($f$) is the $q$-adic Newton polygon of\ \ $\LF^*(f,T)^{(-1)^{n-1}}.$
	Furthermore, the endpoints of NP($f$) and NP($\Delta$) coincide.
\end{Th}
\begin{Def}[ordinary]
	A Laurent polynomial $f$ is called ordinary if NP($f$) = HP($\Delta$). 
\end{Def}

It is clear that the ordinary property of a Laurent polynomial depends on its Newton polyhedron $\Delta$ and on the 
coefficients of $f(x)$. 
Applying the facial decomposition theorem \cite{Wan1993}, we reduce the ordinary property of $f$ to its smaller pieces which are easier to deal with. 
\begin{Th}[Facial decomposition theorem\cite{Wan1993}]\label{thm9}
	Let $f$ be a non-degenerate Laurent polynomial over $\F_q$. Assume $\Delta=\Delta(f)$ is $n$-dimensional and $\delta_1,\ldots, \delta_h$ are all the co-dimension 1 faces of $\Delta$ which don't contain the origin. Let $f^{\delta_i}$ denote the restriction of $f$ to $\delta_i$. Then $f$ is ordinary if and only if $f^{\delta_i}$ is ordinary for $1\leq i\leq h$. 
\end{Th}

\subsection{Boundary decomposition theorems}\label{subsec23}
Before describing the boundary decomposition, we express the L-function in terms of the Fredholm determinant of an infinite Frobenius matrix via Dwork's trace formula.
\subsubsection{Dwork's trace formula}\label{subsubsec232}
Let $\Q_p$ be the field of $p$-adic numbers and $\Omega$ be the completion of $\overline{\Q}_p$. A fixed primitive $p$-th root of unity in $\Omega$ is denoted as $\zeta_p$. Let $\pi$ be the element of $\Q_p(\zeta_p)$  satisfies
\begin{equation*}
	\sum^{\infty}_{m=0}\frac{\pi^{p^m}}{p^m}=0, \ \pi \equiv \zeta_p-1 \mod (\zeta_p-1)^2,  \quad \text{and} \quad \ord_p \pi =\frac{1}{p-1}.
\end{equation*} 
Then, $\pi$ is a uniformizer of $\Q_p(\pi)$ and thus $\Q_p(\pi)=\Q_p(\zeta_p)$. 
Let $E_p(t)$ be the Artin-Hasse exponential series,
\begin{equation*}
	E_p(t)=\exp\left(\sum^{\infty}_{m=0}\frac{t^{p^m}}{p^m}\right)=\sum_{m=0}^{\infty}\lambda_m t^m \in  \Z_p[[x]].
\end{equation*}
In Dwork's terminology, a splitting function $\theta(t)$ is defined to be
\begin{equation*}
	\theta(t)=E_p(\pi t)=\sum_{m=0}^{\infty}\lambda_m\pi^mt^m.
\end{equation*}
A Laurent polynomial $f \in $ $\F_q[x_1^{\pm1}, \ldots, x_n^{\pm1}]$ is written as
\begin{equation*}
	f=\sum_{j=1}^J \bar{a}_j x^{V_j},
\end{equation*}
where $V_j \in {\Z}^n$ and $\bar{a}_j \in \F_q^{*}$. Let $a_j$ be the Teichm\"{u}ller lifting of $\bar{a}_j $ in $\Omega$ satisfying $a_j^q=a_j$. Let
\begin{equation*}
	F(f,x)=\prod_{j=1}^J\theta(a_j x^{V_j})=\sum_{r \in {\Z}^n} F_r(f)x^r.
\end{equation*}
The coefficients are given by 
\begin{equation*}
	F_r(f)=\sum_u (\prod^{J}_{j=1} \lambda_{u_j} a_j^{u_j}) \pi^{u_1+\dots+u_{J}}, \quad r \in {\Z}^n,
\end{equation*}
where the sum is over all the solutions of the following linear system
\begin{equation*}
	\sum^{J}_{j=1}u_jV_j=r \quad \text{with}\quad u_j \in \Z_{\geq 0},
\end{equation*} 
and $\lambda_m$ is $m$-th coefficient of the Artin-Hasse exponential series $E_p(t)$.

Assume $\Delta=\Delta(f)$. Let $L(\Delta)=\Z^{n}\cap C(\Delta)$ be the set of lattice points in the closed cone generated by origin and $\Delta$. For a given point $r\in \R^n$, define the weight function to be
\begin{equation*}
	w(r): =\inf_{\vec{u}}\left\{ \sum_{j=1}^J u_j |\sum_{j=1}^J u_jV_j=r,\quad u_j\in \R_{\geq 0}\right\}.
\end{equation*}
The infinite semilinear Frobenius matrix $A_1(f)$ is the following matrix whose rows and columns are indexed by the lattice points in $L(\Delta)$ with respect to the weights: 
\begin{equation*}
	A_1(f)=(a_{r,s}(f))=(F_{ps-r}(f)\pi^{w(r)-w(s)}),
\end{equation*}
where $r, s\in L(\Delta)$. 
The infinite linear Frobenius matrix $A_a(f)$ is defined to be
\begin{equation*}
	A_a(f)=A_1(f)A_1^{\tau}(f)\cdots A_1^{\tau^{a-1}}(f), 
\end{equation*}
where $\tau$ is the absolute Frobenius automorphism. 

Dwork's trace formula can be expressed in terms of the matrix $A_a(f)$ as follows, see \cite{Wan2004}

\begin{Th}\label{th211}
	We have
	\begin{equation}\label{eq24}
		\LF^{*}(f,T)^{(-1)^{n-1}}=\prod_{i=0}^{n}\det(I-Tq^{i}A_a(f))^{(-1)^i\binom{n}{i}}.
	\end{equation}
	Equivalently,
	\begin{equation}\label{eq25}
		\det(I-TA_a(f))=\prod_{i=0}^{\infty} \left( \LF^{*}(f,q^iT)^{(-1)^{n-1}} \right)^{\binom{n+i-1}{i}}.
	\end{equation}
\end{Th}
Now it suffices to understand the determinant $\det(I-TA_a(f))$.
Based on the fact that $\ord_p F_r(f)\geq \frac{w(r)}{p-1}$, we have the following estimate
\begin{equation*}
	\ord_p( a_{r,s}(f)) \geq \frac{w(ps-r)+w(r)-w(s)}{p-1}\geq w(s).
\end{equation*} 
Let $\xi$ be an element in $\Omega$ satisfying $\xi^D=\pi^{p-1}$. Then $A_1(f)$ can be written in a block form,
\begin{equation*}
	A_1(f)
	=\begin{pmatrix}
		A_{00} &  \xi A_{01}  & \cdots\quad & {\xi}^iA_{0i}&\cdots\\
		A_{10} &  \xi A_{11}  & \cdots\quad & {\xi}^iA_{1i}&\cdots\\
		\vdots & \vdots & \ddots  & \vdots  \\
		A_{i0} &  \xi A_{i1}  & \cdots\quad & {\xi}^iA_{ii}&\cdots\\
		\vdots & \vdots & \ddots  & \vdots
	\end{pmatrix},
\end{equation*}
where the block $A_{ii}$ is a $p$-adic integral $W_{\Delta}(i) \times W_{\Delta}(i)$ matrix.
This implies that the $q$-adic Newton polygon of $\det(I-TA_1(f))$ has a natural lower bound which can be identified with the chain level version of the Hodge polygon.

\begin{Def}
	Let $P(\Delta)$ be the polygon in $\R^2$ with vertices $(0,0)$ and 
	\begin{equation*}
		P_k=\left( \sum^k_{m=0}W_{\Delta}(m), \frac{1}{D}\sum^k_{m=0}m W_{\Delta}(m) \right), \quad  k=0,1,2, \ldots
	\end{equation*}
\end{Def}

The chain level version of Adolphson-Sperber's lower bound and the ordinary property are as follows.

\begin{Prop}[\cite{AS1987}] 
	The $q$-adic Newton polygon of $\det(I-TA_a(f))$ lies above $P(\Delta).$
\end{Prop}

\begin{Prop}[\cite{Wan2004}]\label{prop c}
	Notations as above. Assume $f$ is non-degenerate with $\Delta=\Delta(f)$. Then $\mathrm{NP}(f)=\mathrm{HP}(\Delta)$ if and only if the $q$-adic Newton polygon of $\det(I-TA_a(f))$ coincides with its lower bound $P(\Delta).$
\end{Prop}

\subsubsection{Boundary decomposition}\label{subsubsec233}
Let $f \in $ $\F_q[x_1^{\pm1}, \ldots, x_n^{\pm1}]$ with $\Delta=\Delta(f)$, where $\Delta$ is an $n$-dimensional integral convex polyhedron in $\R^n$ containing the origin. Let $C(\Delta)$ be the cone generated by $\Delta$ in $\R^n.$ 
\begin{Def}\label{defbd}
	The boundary decomposition 
	\begin{equation*}
		B(\Delta)=\{ \text{ the interior of a closed face in }C(\Delta) \text{ containing the origin} \}
	\end{equation*}
	is the unique interior decomposition of $C(\Delta)$ into a disjoint union of relatively open cones. 
\end{Def}
If the origin is a vertex of $\Delta$, then it is the unique 0-dimensional open cone in $B(\Delta)$. Recall that $A_1(f)=(a_{r,s}(f))$ is the infinite semilinear Frobenius matrix whose rows and columns are indexed by the lattice points in $L(\Delta)$. For $\Sigma \in B(\Delta)$, we define $A_1(\Sigma,f)$ to be the submatrix of $A_1(f)$ with $r,s \in \Sigma$. Let $f^{\overline{\Sigma}}$ be the restriction of $f$ to the closure of $\Sigma$. Then $A_1(\Sigma,f^{\overline{\Sigma}})$ denotes the submatrix of $A_1(f^{\overline{\Sigma}})$ with $r,s \in \Sigma$.

Let $B(\Delta)=\{\Sigma_0, \ldots, \Sigma_h\}$ such that $\text{dim}(\Sigma_i)\leq \text{dim}(\Sigma_{i+1})$, $i=0,\ldots,h-1.$ Define $B_{ij}=(a_{r,s}(f))$ with $ r\in \Sigma_i$ and $ s\in \Sigma_j$ $(0\leq i,j \leq h)$. After a permutation of basis vectors, the infinite semilinear Frobenius matrix can be written as
\begin{equation}
	A_1(f)=
	\begin{pmatrix}
		B_{00} &  B_{01}  & \cdots\quad &B_{0h}\\
		B_{10} &  B_{11}  & \cdots\quad & B_{1h}\\
		\vdots & \vdots & \ddots  & \vdots  \\
		B_{h0} & B_{h1}  & \cdots\quad & B_{hh}
	\end{pmatrix},
\end{equation} 
where $B_{ij}=0$ for $i>j$. Then $\det(I-TA_1(f))=\prod_{i=0}^h\det(I-TB_{ii})$ and we have the boundary decomposition theorem.
\begin{Th}[Boundary decomposition\cite{Wan1993}] \label{th215}
	Let $f\in \F_q[x_1^{\pm1}, \ldots, x_n^{\pm1}]$ with $\Delta=\Delta(f)$. Then we have the following factorization
	\begin{equation*}
		\det(I-TA_1(f))=\prod_{\Sigma \in B(\Delta)}\det \left(I-TA_1(\Sigma,f^{\overline{\Sigma}})\right).
	\end{equation*}
\end{Th}

\subsection{Diagonal local theory}\label{subsec24}
In this subsection, we introduce some non-degenerate and ordinary criteria when the Laurent polynomial is diagonal.
\begin{Def}
	A Laurent polynomial $f \in  \F_{q}[x_1^{\pm1},\ldots, x_n^{\pm1}]$ is called diagonal if $f$ has exactly $n$ non-constant terms and $\Delta(f)$ is an $n$-dimensional simplex in $\R^n.$
\end{Def}
Let $f$ be a diagonal Laurent polynomial over  $\F_{q}$. Write 
\begin{equation*}
	f(x_1,x_2, \ldots x_n)=\sum_{j=1}^n a_j x^{V_j},
\end{equation*}
where $a_j \in \F ^*_{q}$ and $V_j=(v_{1j},\ldots, v_{nj})\in \Z^n$ for $1\leq j\leq n$. Let $\Delta=\Delta(f)$. The vertex matrix of $\Delta$ is defined to be
\begin{equation*}
	M(\Delta)=(V_1,\ldots,V_n),
\end{equation*} 
where the $i$-th column is the $i$-th exponent of $f$. Since $f$ is diagonal, $M(\Delta)$ is invertible.
\begin{Prop}\label{propc1}
	Suppose $f\in \F_{q}[x_1^{\pm1},\ldots, x_n^{\pm1}]$ is diagonal with $\Delta=\Delta(f)$. Then $f$ is non-degenerate if and only if $p$ is relatively prime to $\det(M(\Delta))$. 
\end{Prop}
Let $S(\Delta)$ be the solution set of the following linear system
\begin{align*}
	M(\Delta)
	\begin{pmatrix}
		r_1\\ r_2 \\ \vdots \\ r_n
	\end{pmatrix}
	\equiv 0 ~(\bmod1),\quad r_i \in \Q \cap [0,1).
\end{align*}
It's easy to prove that $S(\Delta)$ is an abelian group and its order is given by
\begin{align}\label{eq2}
	\left|\det{M(\Delta)}\right|=n!\Vol(\Delta).
\end{align}
Let $S_p(\Delta)$ denote the prime to $p$ part of $S(\Delta)$. It is an abelian subgroup of order equal to the prime to $p$ factor of $\det{M(\Delta)}$. In particular, $S_p(\Delta)=S(\Delta)$ if $p$ is relatively prime to $\det{M(\Delta)}$.
 By the Stickelberger theorem for Gauss sums, we have the following ordinary criterion for a non-degenerate Laurent polynomial\cite{Wan2004}.
\begin{Prop}\label{prop15}
A diagonal Laurent polynomial $f$ is ordinary at $p$ if and only if the norm function $|r|=r_1+\cdots +r_n$ on $S_{p}(\Delta)$ is stable under the p-action: That is, for each $r\in S_{p}(\Delta)$, we have
$|r|=|\!\left\{pr\right\}\!|$, where $\left\{pr\right\}$ is the class of $pr$ in $S_{p}(\Delta)$.
\end{Prop}

\section{Proof of the Main Theorems}\label{sec3}
We prove the main theorems in this section. 

\subsection{Proof of Theorem \ref{thm0}}\label{subsec31}
Recall that for integer $n\geq 1$, the twisted inverted $n$-variable Kloosterman sum is defined to be
\begin{equation*}
	S_n(\chi,b)=\mathop{\sum_{x_1\cdots x_{n+1}=b}}_{x_1+\cdots+x_{n+1}\neq 0} \chi_1(x_1)\cdots\chi_{n+1}(x_{n+1})\psi\left(\frac{1}{x_1+\cdots+x_{n+1}}\right),
\end{equation*}
where $b\in\F_{q}^*$, $\psi: \F_{q} \rightarrow \mathbb{C}^*$ is a nontrivial additive character and $\chi_1,\ldots,\chi_{n+1}: \F_{q}^* \rightarrow \mathbb{C^*}$ are multiplicative characters.
Let $\chi: \F_{q}^* \rightarrow \mathbb{C^*}$ denote a multiplicative character. By the orthogonality of characters, we have
\begin{align}\label{equs}
	S_n(\chi,b)=&\frac{1}{q(q-1)}\sum_{\lambda, x_i\in\F^*_q} \sum_{u\in\F_q}\psi\left(u\left(x_1+\cdots+x_{n+1}-\lambda\right)\right) \chi_1(x_1)\cdots\chi_{n+1}(x_{n+1})\nonumber\\
	&\times\psi\left(\frac{1}{\lambda}\right)\sum_{\chi}\chi\left(\frac{x_1\cdots\chi_{n+1}}{b}\right)\nonumber\\
	=&\frac{1}{q(q-1)}\sum_{\lambda\in\F^*_q}\sum_{x_i\in\F^*_q}
	\chi_1(x_1)\cdots\chi_{n+1}(x_{n+1})
	\psi\left(\frac{1}{\lambda}\right)\sum_{\chi}\chi\left(\frac{x_1\cdots\chi_{n+1}}{b}\right)\nonumber\\
	&+\frac{1}{q(q-1)}\sum_{\lambda\in\F^*_q}\sum_{x_i\in\F^*_q} \sum_{u\in\F_q^*}\psi\left(u\left(x_1+\cdots+x_{n+1}-\lambda\right)\right) \nonumber\\
	&\quad\times\chi_1(x_1)\cdots\chi_{n+1}(x_{n+1})\psi\left(\frac{1}{\lambda}\right)\sum_{\chi}\chi\left(\frac{x_1\cdots\chi_{n+1}}{b}\right)\nonumber\\
	=&S_1+S_2.
\end{align}
Then
\begin{align}\label{equs1}
	S_1&=\frac{1}{q(q-1)}\sum_{\lambda\in\F^*_q}\psi\left(\frac{1}{\lambda}\right)\sum_{\chi}\chi^{-1}(b)\sum_{x_i\in\F^*_q}
	\left(\chi\chi_1\right)(x_1)\cdots\left(\chi\chi_{n+1}\right)(x_{n+1})\nonumber\\
	&=\frac{1}{q(q-1)}\sum_{\lambda\in\F^*_q}\psi\left(\frac{1}{\lambda}\right)
	\sum_{\chi}\chi^{-1}(b)\prod^{n+1}_{i=1}\left(\sum_{x_i\in\F^*_q}\left(\chi\chi_i\right)(x_i)\right)\nonumber\\
	&=\begin{cases}
		\displaystyle -\frac{(q-1)^{n}}{q}\chi_1(b), & \text{if}\ \chi_1=\cdots=\chi_{n+1},\\
		\displaystyle 0, & \text{otherwise}.
	\end{cases}
\end{align}
If $\chi$ is trivial, the Gauss sum $G(\chi)=-1$. If $\chi$ is non-trivial, $|G(\chi)|=\sqrt{q}$. Then
\begin{align}\label{equs2}
	S_2&=\frac{1}{q(q-1)}\sum_{\lambda,u\in\F^*_q}\sum_{\chi}\chi^{-1}(b)\sum_{x_i\in\F^*_q}
	\left(\chi\chi_1\right)(x_1)\psi(ux_1)\cdots\left(\chi\chi_{n+1}\right)(x_{n+1})\psi(ux_{n+1})\nonumber\\
	&\quad\times\psi(-u\lambda)\psi\left(\frac{1}{\lambda}\right)\nonumber\\
	&=\frac{1}{q(q-1)}\sum_{\lambda,u\in\F^*_q}\sum_{\chi}\chi^{-1}(b)\overline{\chi^{n+1}\chi_1\cdots\chi_{n+1}}(u)
	\psi(-u\lambda)\psi\left(\frac{1}{\lambda}\right)
	G(\chi\chi_1)\cdots G(\chi\chi_{n+1})\nonumber\\
	&=\frac{1}{q(q-1)}\sum_{\chi}\chi^{-1}(b)\left(\sum_{\lambda\in\F^*_q}
	\overline{\chi^{n+1}\chi_1\cdots\chi_{n+1}}\left(-\frac{1}{\lambda}\right)\psi\left(\frac{1}{\lambda}\right)\right)
	G(\overline{\chi^{n+1}\chi_1\cdots\chi_{n+1}})\nonumber\\
	&\quad\times G(\chi\chi_1)\cdots G(\chi\chi_{n+1})\nonumber\\
	&=\frac{1}{q(q-1)}\sum_{\chi}\chi^{-1}(b)\chi^{n+1}\chi_1\cdots\chi_{n+1}(-1)
	G(\overline{\chi^{n+1}\chi_1\cdots\chi_{n+1}})G(\overline{\chi^{n+1}\chi_1\cdots\chi_{n+1}})\nonumber\\ 
	&\quad\times G(\chi\chi_1)\cdots G(\chi\chi_{n+1}).
\end{align}
Since $|G(\chi)| \leq \sqrt{q}$, it follows that $|S_2|\leq q^{\frac{n+1}{2}}$. Combining (\ref{equs}) and (\ref{equs1}), we can deduce the following bounds. 
\begin{align*}
	\left|S_n(\chi,b)+\frac{(q-1)^{n}}{q}\chi_1(b)\right| \leq q^{\frac{n+1}{2}}, \quad \text{if}\ \chi_1=\cdots=\chi_{n+1},
\end{align*}
and
\begin{align*}
	\left|S_n(\chi,b)\right| \leq q^{\frac{n+1}{2}}, 
	\quad \text{if}\ \chi_i\neq\chi_j \ \text{for some}\ i\neq j.
\end{align*}
This proves Theorem \ref{thm0}.

\subsection{Proof of Theorem \ref{thm2}}\label{subsec32}
The twisted inverted Kloosterman sum $S_{n}(\chi,b)$ has the expression
\begin{align}\label{equ32}
	S_{n}(\chi,b)&=\mathop{\sum_{x_1+\cdots+x_{n}+\frac{b}{x_1\cdots x_{n}}\neq 0}}_{x_i\in\F_q^*} \chi_1(x_1)\cdots\chi_{n}(x_n)\chi_{n+1}\left(\frac{b}{x_1\cdots x_{n}}\right)\nonumber\\
	&\quad\times\psi\left(\frac{1}{x_1+\cdots+x_{n}+\frac{b}{x_1\cdots x_{n}}}\right) \nonumber\\
	&=\mathop{\sum_{z\left(x_1+\cdots+x_{n}+\frac{b}{x_1\cdots x_{n}}\right)=1}}_{z,\ x_i\in\F_q^*} \chi_{n+1}(b)\left(\chi_1\overline{\chi_{n+1}}\right)\!(x_1)\cdots\left(\chi_n\overline{\chi_{n+1}}\right)\!(x_n)\psi\left(z\right)\nonumber\\
	&=\frac{1}{q}\mathop{\sum_{z,\ x_i\in\F_q^*}}_{y\in\F_q} \chi_{n+1}(b)
	\left(\chi_1\overline{\chi_{n+1}}\right)\!(x_1)\cdots\left(\chi_n\overline{\chi_{n+1}}\right)\!(x_n)\nonumber\\
	&\quad\times\psi\left(z+y\left(1-z\left(x_1+\cdots+x_{n}+\frac{b}{x_1\cdots x_{n}}\right)\right)\right)\nonumber\\
	&=\frac{\chi_{n+1}(b)}{q}\left(\sum_{z,\ x_i\in\F_q^*}\left(\chi_1\overline{\chi_{n+1}}\right)\!(x_1)\cdots\left(\chi_n\overline{\chi_{n+1}}\right)\!(x_n)\psi\left(z\right)+E_n(\chi,b)\right)\nonumber\\
	&=\begin{cases}
		\displaystyle-\frac{(q-1)^n}{q}\chi_1(b)+\frac{1}{q}\chi_{1}(b)E_n(\chi,b), & \text{if}\ \chi_1=\cdots=\chi_{n+1},\\
		\displaystyle\frac{1}{q}\chi_{n+1}(b)E_n(\chi,b), & \text{if}\ \chi_i\neq \chi_j\ \text{for some}\ i\neq j,
	\end{cases}
\end{align}
where 
\begin{align*}
	E_n(\chi,b)=&\sum_{y, z, x_i\in\F_q^*}\left(\chi_1\overline{\chi_{n+1}}\right)\!(x_1)\cdots
	\left(\chi_n\overline{\chi_{n+1}}\right)\!(x_n)\nonumber\\
	&\times\psi\left(z+y\left(1-z\left(x_1+\cdots+x_{n}+\frac{b}{x_1\cdots x_{n}}\right)\right)\right).
\end{align*}
In order to prove Theorem \ref{thm2}, it suffices to estimate $E_n(\chi,b)$.

Let $f\in \F_{q}[x_1^{\pm1},\ldots, x_{n+2}^{\pm1}]$ be the Laurent polynomial defined by  
\begin{align*}
	f(x_1,\cdots,x_{n+2})=x_{n+1}\left(1-x_{n+2}\left(x_1+\cdots+x_{n}+\frac{b}{x_1\cdots x_{n}}\right)\right)+x_{n+2}.
\end{align*}
As defined in (\ref{equ1}), $E_n(\chi,b)$ is the twisted toric exponential sum associated to $f$.
Let $\Delta=\Delta(f)$ denote the Newton polyhedron corresponding to $f$.  Clearly, $\dim\Delta=n+2$ and $\Delta$ has $n+4$ vertices in $\R^{n+2}$: $V_0 = (0,\cdots,0)$(the origin), $V_1 = (1,0,\cdots,0,1,1)$, $V_2 = (0,1,\cdots,0,1,1)$, \ldots, $V_{n} = (0,0,\cdots,1,1,1)$, $V_{n+1} = (-1,\cdots,-1,1,1)$, $V_{n+2} = (0,\cdots,0,1,0)$ and $V_{n+3} = (0,\cdots,0,0,1)$.
Furthermore, $\Delta$ has exactly 2 co-dimension 1 faces not containing the origin. Explicitly, they are
\begin{align*}
	\delta_1: x_{n+1}=1 \quad \text{and} \quad \delta_2: x_{n+2}=1.
\end{align*}
Vertices $V_1,\ldots,V_{n+2}$ determine the face $\delta_1$ and vertices $V_1,\ldots,V_{n+1},V_{n+3}$ determine the face $\delta_2$. 	Let $M(\delta_i)$ be the vertex matrix of $\delta_i$, we have
\begin{align}\label{eqd}
	M(\delta_1)=
	\begin{pmatrix}
		1 & 0 & \cdots & 0 & -1 & 0\\ 
		0 & 1 & \cdots & 0 & -1 & 0\\ 
		\vdots & \vdots & \ddots & \vdots & \vdots & \vdots\\ 
		0 & 0 & \cdots & 1 & -1 & 0\\ 
		1 & 1 & \cdots & 1 & 1 & 1\\ 
		1 & 1 & \cdots & 1 & 1 & 0
	\end{pmatrix},
	\qquad
	M(\delta_2)=
	\begin{pmatrix}
		1 & 0 & \cdots & 0 & -1 & 0\\ 
		0 & 1 & \cdots & 0 & -1 & 0\\ 
		\vdots & \vdots & \ddots & \vdots & \vdots & \vdots\\ 
		0 & 0 & \cdots & 1 & -1 & 0\\ 
		1 & 1 & \cdots & 1 & 1 & 0\\ 
		1 & 1 & \cdots & 1 & 1 & 1
	\end{pmatrix}.
\end{align}
Explicitly, each $f^{\delta_i}$ is diagonal for $i=1,2$. The restriction of $f$ to $\delta_i$ is defined by 
\begin{align*}
	f^{\delta_i}=\sum_{V_j\in\delta_i}a_jx^{V_j}.
\end{align*}

\begin{figure}[htbp] 
	\centering 	
	\begin{tikzpicture}
		\node[right] (0) at (4.3,0) {$V_0$};
		\node[above] (1) at (0,3.3) {$V_1$};
		\node[right] (2) at (7,4.3) {$V_2$};
		\node[above] (3) at (3.5,2.5) {$V_3$};
		\node[above] (4) at (4.3,1.5) {$V_4$};
		\fill[fill=blue] (4.3,0) circle (2pt);
		\fill[fill=blue] (0,3.3) circle (2pt);
		\fill[fill=blue] (7,4.3) circle (2pt);
		\fill[fill=blue] (3.5,2.5) circle (2pt);
		\fill[fill=blue] (4.3,1.5) circle (2pt);
		\draw (0,3.3)--(4.3,0)--(7,4.3)--(0,3.3);
		\draw (0,3.3)--(4.3,1.5)--(7,4.3) ;
		\draw (4.3,0)--(4.3,1.5) ;
		\draw [dashed] (0,3.3)--(3.5,2.5)--(7,4.3) ;
		\draw [dashed] (4.3,0)--(3.5,2.5) ;
	\end{tikzpicture}
	\caption{$\Delta$ for $n=1$} 
\end{figure}
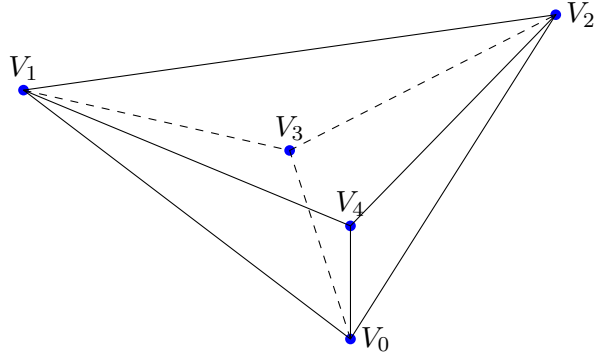

\begin{Prop}\label{th5}
	\begin{enumerate}[(i).]
		\item The denominator $D=1$.
		\item $f$ is non-degenerate if and only if $p\nmid (n+1)$.
		\item $\displaystyle\Vol(\Delta)=\frac{2n+2}{(n+2)!}$.
	\end{enumerate}
\end{Prop}
\begin{proof}
	The denominator $D=1$ can be deduced immediately from the equation of $\delta_i$.
	Since $\delta_1$ and $\delta_2$ are the co-dimension 1 faces of $\Delta(f)$ not containing the origin, it suffices to prove $f^{\delta_1}$ and $f^{\delta_2}$ are non-degenerate.
	By Proposition \ref{propc1}, $f^{\delta_i}$ is non-degenerate if and only if $p$ is relatively prime to $\det(M(\delta_i))$. By formula (\ref{eqd}), 
	\begin{align}\label{eqd1}
		\det(M(\delta_1))=-(n+1)
		\quad \text{and} \quad
		\det(M(\delta_2))=n+1.
	\end{align}
	This proves $(ii)$. 
	
	Let $\Delta_i$ be the polytope generated by $\delta_i$ and the origin. The facial decomposition of $\Delta$ 
	implies that
	\begin{align*}
		\Vol(\Delta)=\Vol(\Delta_1)+\Vol(\Delta_2).
	\end{align*}
	By formula (\ref{eq2}) and (\ref{eqd1}), we obtain $(iii)$. 
\end{proof}

Combining Theorem \ref{thm3} with Proposition \ref{th5}, if $p\nmid (n+1)$, we have
\begin{align}\label{eque}
	|E_n(\chi,b)|\leq (n+2)!\Vol(\Delta)\cdot q^{\frac{n+2}{2}}=2(n+1)q^{\frac{n+2}{2}},
\end{align}
where $p$ is the characteristic of $\F_q$. 
Putting (\ref{equ32}) and (\ref{eque}) together, we then obtain the following bounds when $p\nmid(n+1)$.
\begin{align*}
	|S_{n}(\chi,b)+\frac{(q-1)^n}{q}\chi_1(b)|\leq 2(n+1)q^{\frac{n}{2}}, \quad \text{if}\ \chi_1=\cdots=\chi_{n+1},
\end{align*}
and
\begin{align*}
	|S_{n}(\chi,b)|\leq 2(n+1)q^{\frac{n}{2}}, \quad \text{if}\ \chi_i\neq \chi_j\ \text{for some}\ i\neq j.
\end{align*}

In the case $\chi_1=\cdots=\chi_{n+1}$, the twisted sum $E_n(\chi, b)$ 
becomes the following untwisted toric exponential sum
\begin{align*}
	E_n(\chi,b)=&\sum_{y, z, x_i\in\F_q^*}\psi\left(z+y\left(1-z\left(x_1+\cdots+x_{n}+\frac{b}{x_1\cdots x_{n}}\right)\right)\right).
\end{align*}
Since the origin is a vertex of $\Delta$ and the polynomial inside the additive character has no constant term, $1$ is a trivial 
eigenvalue of the middle dimensional cohomology. Removing this trivial eigenvalue from the error term, one gets 
\begin{align*}
	|E_{n}(\chi,b) -(-1)^{n+2}|\leq (2n+1)q^{\frac{n}{2}}, \quad \text{if}\ \chi_1=\cdots=\chi_{n+1}, 
\end{align*}
and hence the 
slightly sharper estimate
\begin{align*}
	|S_{n}(\chi,b)+\frac{(q-1)^n+(-1)^{n+1}}{q}\chi_1(b)|\leq (2n+1)q^{\frac{n}{2}}, \quad \text{if}\ \chi_1=\cdots=\chi_{n+1}. 
\end{align*}
This proves Theorem \ref{thm2}.

\subsection{Proof of Theorem \ref{thm1}}\label{subsec33}
Similar to formula (\ref{equ32}), we relate the untwisted inverted Kloosterman sum $S_{k,n}(b)$ to toric exponential sum $S_k^*(f)$.
\begin{align}\label{eq32}
	S_{k,n}(b)&=\mathop{\sum_{x_1+\cdots+x_{n}+\frac{b}{x_1\cdots x_{n}}\neq 0}}_{x_i\in\F_{q^k}^*} \psi\left(\Tr_k\left(\frac{1}{x_1+\cdots+x_{n}+\frac{b}{x_1\cdots x_{n}}}\right)\right) \nonumber\\
	&=\mathop{\sum_{z\left(x_1+\cdots+x_{n}+\frac{b}{x_1\cdots x_{n}}\right)=1}}_{z,\ x_i\in\F_{q^k}^*} \psi\left(\Tr_k\left(z\right)\right)\nonumber\\
	&=\frac{1}{q^k}\mathop{\sum_{z,\ x_i\in\F_{q^k}^*}}_{y\in\F_{q^k}}
	\psi\left(\Tr_k\left(z+y\left(1-z\left(x_1+\cdots+x_{n}+\frac{b}{x_1\cdots x_{n}}\right)\right)\right)\right)\nonumber\\
	&=-\frac{(q^k-1)^{n}}{q^k}+\frac{1}{q^k}S_k^*(f).
\end{align}
where $f$ is the Laurent polynomial given by
\begin{align*}
	f(x_1,\cdots,x_{n+2})=x_{n+1}\left(1-x_{n+2}\left(x_1+\cdots+x_{n}+\frac{b}{x_1\cdots x_{n}}\right)\right)+x_{n+2}
\end{align*}
and
\begin{align*}
	S_k^*(f)=&\sum_{x_i\in\F_{q^k}^*}
	\psi\left(\Tr_k\left(x_{n+2}+x_{n+1}\left(1-x_{n+2}\left(x_1+\cdots+x_{n}+\frac{b}{x_1\cdots x_{n}}\right)\right)\right)\right).
\end{align*}

The L-functions associated to $S_{k,n}(b)$ and $S_k^*(f)$ are defined as
\begin{align*}
	\LF_{n}(b,T)=\exp \left(\sum^\infty _ {k=1} S_{k,n}(b) \frac{T^k}{k} \right)
	\quad\text{and}\quad 
	\LF^*(f,T)=\exp \left(\sum^\infty _ {k=1} S^*_k (f) \frac{T^k}{k} \right).
\end{align*}
It follows from formula (\ref{eq32}) that 
\begin{align}\label{eq34}
	\LF_{n}(b,T)&=\exp \left(\sum^\infty _ {k=1} -\frac{\left(q^k-1\right)^n\cdot T^k}
	{q^k\cdot k} \right)	\LF^*\left(f,T/q\right)\nonumber\\
	&=\prod_{i=0}^{n} \exp \left((-1)^{n-i+1}\binom{n}{i}\sum^\infty _ {k=1} \frac{\left(q^{i-1}T\right)^k}{k} \right)	\LF^*\left(f,T/q\right)\nonumber\\
	&=\LF^*\left(f,T/q\right)
	\prod_{i=0}^{n}\left(\frac{1}{1-q^{i-1}T}\right)^{(-1)^{n-i+1}\binom{n}{i}} .
\end{align}

The main purpose of this subsection is to determine the slopes and weights of \ $\LF_{n}(b,T)$. Based on formula (\ref{eq34}), it suffices to consider \ $\LF^*(f,T)$ instead. Let $\Delta=\Delta(f)$ denote the Newton polyhedron corresponding to $f$. Some of the geometric properties about $\Delta$ have been discussed in subsection \ref{subsec32}. In Proposition \ref{th5}, we proved that $f$ is non-degenerate if and only if $p\nmid (n+1)$. In this case, the L-function $\LF^*(f,T)^{(-1)^{n+1}}$ is a polynomial of degree $2n+2$. To determine the slopes of the reciprocal roots of $\LF^*(f,T)^{(-1)^{n+1}}$, we shall compute the Hodge polygon and consider when it coincides with the Newton polygon.

\begin{Prop}\label{propo}
	The Laurent polynomial $f$ is ordinary if and only if $p\equiv 1\bmod (n+1)$.
\end{Prop}
\begin{proof}
	By facial decomposition theorem, it suffices to consider $f^{\delta_i}$ for $i=1,2$. 
	Let $S(\delta_i)$ be the solution set of the following linear system
	\begin{align}\label{eqs1}
		M(\delta_i)
		\begin{pmatrix}
			r_1\\ r_2 \\ \vdots \\ r_{n+2}
		\end{pmatrix}
		=u\in\Z^{n+2},\quad\text{where}\ r_j \in \Q \cap [0,1).
	\end{align}
	 For $i=1$ and a given point $u=(x_1,\ldots,x_{n+2})^T$, linear system (\ref{eqs1}) equals to 
\begin{align}\label{eqs2}
	\begin{cases}
		x_1=r_1-r_{n+1},\\
		x_2=r_2-r_{n+1},\\
		\cdots\\
		x_{n}=r_{n}-r_{n+1},\\
		x_{n+1}=r_1+\cdots+r_{n+2},\\
		x_{n+2}=r_1+\cdots+r_{n+1},
	\end{cases}\quad\text{where}\ r_j \in \Q \cap [0,1).
\end{align}
	Note that $x_j\in\Z$, where $1\leq j\leq n+2$. For any $r=(r_1,\ldots,r_{n+2})^T\in S(\delta_1)$, we have
	\begin{align*}
		r_1=\cdots=r_{n}=r_{n+1}\in\frac{\Z}{n+1} \quad\text{and}\quad r_{n+2}=0.
	\end{align*}
	Let $S_p(\delta_i)$ denote the prime to $p$ part of $S(\delta_i)$. In particular, $S_p(\delta_i)=S(\delta_i)$ if $p\nmid\det(M(\delta_i))$.
	Suppose $p\nmid (n+1)$, the norm function $|r|$ and $|\{pr\}|$ are given by 
	\begin{align*}
		|r|=(n+1)r_1 \quad\text{and}\quad |\{pr\}|=(n+1)\{pr_1\}.
	\end{align*}
	Then $|r|$ on $S_p(\delta_1)$ is stable under the $p$-action if and only if $p\equiv 1 \bmod (n+1)$. 
	To see this, it suffices to consider the unique point $r=(\frac{1}{n+1}, \cdots, \frac{1}{n+1}, 0)$ with 
	norm $1$. This condition holds for $S_p(\delta_2)$ through a similar proof. 
	By Proposition \ref{prop15}, we obtain Proposition \ref{propo}.
\end{proof}

\begin{Th}\label{th6}
The $n+2$ Hodge numbers of $\Delta$ are $\{ 1, 2, 2, \cdots, 2, 1\}$. Namely, 
\begin{align*}
	H_{\Delta}(0)=1,\ H_{\Delta}(1)=\cdots=H_{\Delta}(n)=2,\ H_{\Delta}(n+1)=1.
\end{align*}

\end{Th}
\begin{proof}
Let $\Delta_i$ be the polytope generated by the origin and $\delta_i$. Let $u=(x_1,\ldots,x_{n+2})^T\in C(\Delta_i)$ be a lattice point with the weight $w(u)=k$, where $0\leq k\leq n+2$.
For $i=1,2$, consider the linear system (\ref{eqs1}).
Since $f^{\delta_i}$ is diagonal, system (\ref{eqs1}) has a unique solution $r=(r_1,\ldots,r_{n+2})^T$ for a fixed point $u\in C(\Delta_i)$. In this case, the weight is given by 
\begin{align*}
	w(u)=r_1+\cdots+r_{n+2}=|r|.
\end{align*}
When $i=1$, the linear equations (\ref{eqs2}) has exact one solution $u=(0,\ldots,0,k,k)^T$. Since $x_{n+2}=\sum_{j=1}^{n+1}r_j=k$ and $0\leq r_j<1$, we get the restriction $0\leq k< n+1$. The Hodge number $H_{\Delta_1}(k)$ counts the number of lattice points $u$ of weight $k/D$ in a fundamental domain: That is,
\begin{align*}
	H_{\Delta_1}(k)=
\begin{cases}
	1, & \text{for}\ \ 0\leq k< n+1,\\
	0, & \text{for}\ \ k\geq n+1.
\end{cases}	
\end{align*}
The generating function of $H_{\Delta_1}(k)$ is
\begin{align*}
	H_1(x)=1+x+\cdots+x^{n}.
\end{align*}
By formula (\ref{eq3}), we get the generating function of $W_{\Delta_1}(k)$ as follow.
\begin{align*}
	W_1(x)=\sum^{\infty}_{k=0}W_{\Delta_1}(k)x^k=\frac{H_1(x)}{(1-x)^{n+2}}
	=\frac{1-x^{n+1}}{(1-x)^{n+3}}.
\end{align*}
Let $H_2(x)$ and $W_2(x)$ be the generating function of $H_{\Delta_2}(k)$ and $W_{\Delta_2}(k)$, respectively.
Similarly, we can prove $H_2(x)=H_1(x)$ and $W_2(x)=W_1(x)$.
The polytope $\Delta_1\bigcap\Delta_2$ is determined by $V_1,\ldots,V_{n+1}$, whose
generating function is given by
\begin{align*}
	W_3(x)=\sum^{\infty}_{k=0}W_{\Delta_1\bigcap\Delta_2}(k)x^k
	=\frac{1-x^{n+1}}{(1-x)^{n+2}}.
\end{align*}
By facial decomposition, we have
\begin{align}\label{eqw}
	W_{\Delta}(k)=W_{\Delta_1}(k)+W_{\Delta_2}(k)-W_{\Delta_1\bigcap\Delta_2}(k),
\end{align}
which implies 
\begin{align*}
	W(x)=\sum^{\infty}_{k=0}W_{\Delta}(k)x^k=W_1(x)+W_2(x)-W_3(x)
	=\frac{1+2x+\cdots+2x^{n}+x^{n+1}}{(1-x)^{n+2}}.
\end{align*}
This gives the Hodge numbers of $\Delta$ via formula (\ref{eq3}), that is, 
\begin{align*}
	H_{\Delta}(0)=1,\ H_{\Delta}(1)=\cdots=H_{\Delta}(n)=2,\ H_{\Delta}(n+1)=1.
\end{align*}
\end{proof}

When $f$ is ordinary, the slopes of $\LF^*(f,T)^{(-1)^{n+1}}$ can be deduced from Theorem \ref{th6}.
\begin{Th}\label{thm34}
	If $p\equiv 1\bmod (n+1)$, the slope sequence of\ \ $\LF^*(f,T)^{(-1)^{n+1}}$ is given by 
	$$\{0, 1, 1,2, 2, \ldots, n, n, n+1\}.$$
\end{Th}
\begin{proof}
	This theorem follows from Lemma \ref{le4}, Proposition \ref{propo} and Theorem \ref{th6}.
\end{proof}
Note that the converse of this theorem is also true, as Proposition \ref{propo} shows that the condition 
$p\equiv 1\bmod (n+1)$ is a necessary and sufficient condition for $f$ to be ordinary. 

Now we are ready to consider the weights for the reciprocal roots of $\LF^*(f,T)^{(-1)^{n+1}}$.
\begin{Th}\label{thm35}
	Suppose $p\nmid (n+1)$. We have
	\begin{align*}
		\LF^*(f,T)^{(-1)^{n+1}}&=(1-T)(1-qT)\prod^{2n}_{i=1}(1-\beta_iT).
	\end{align*}
	For each $1\leq i\leq 2n$, the reciprocal root $\beta_i$ satisfies $|\beta_i|= q^{\frac{n+2}{2}}$. 
\end{Th}
\begin{proof}
	Since the origin is a vertex of $\Delta$, we decompose the cone $C(\Delta)$ via boundary decomposition $B(\Delta)$. Let $N(i)$ be the number of $i$-dimensional face $\Sigma_{i}$ of $C(\Delta)$, where $0\leq i\leq \mathrm{dim}\Delta$. For Newton polyhedron $\Delta=\Delta(f)$, we have $N(0)=1$ and $N(1)=n+3$. 
	Note that $\Sigma_i$ is an open cone and $\Sigma_i\in B(\Delta)$. Let $\overline{\Sigma}_i$ be the closure of $\Sigma_i$.
	For simplicity, we denote the Fredholm determinants as 
	\begin{align*}
		D(T)&=\mathrm{det}\left(I-TA_1(f)\right),\\
		D^{\circ}_i(T)&=\mathrm{det}\left(I-TA_1(\Sigma_i,f^{\overline{\Sigma}_i})\right),\\
		D_i(T)&=\mathrm{det}\left(I-TA_1(\overline{\Sigma}_i,f^{\overline{\Sigma}_i})\right).
	\end{align*}	
	The unique $0$-dimensional cone $\overline{\Sigma}_0$ is the origin
	and $D_0(T)=D^{\circ}_0(T)=1-T$. When $i=1$, each $f^{\overline{\Sigma}_1}$ can be normalized to $x$ by variable substitution. That is,
	\begin{align*}
		\LF^*(f^{\overline{\Sigma}_1},T)
		=\exp \left(\sum^\infty _ {k=1}-\frac{T^k}{k}\right)=1-T.
	\end{align*}
	By formula (\ref{eq25}), we have
	\begin{align*}
		D_1(T) &= \prod_{i=0}^{\infty} \left( \LF^{*}\left(f^{\overline{\Sigma}_1},q^iT\right)\right)^{\binom{i}{i}}
		= (1-T)\left(1-qT\right)\left(1-q^2T\right)\cdots.
	\end{align*} 
	Since the only boundary of $\overline{\Sigma}_1$ are ${\Sigma}_1$ and $\overline{\Sigma}_0$, we get $D^{\circ}_1(T)$ after eliminating $D^{\circ}_0(T)$, i.e., 
	\begin{align*}
		D^{\circ}_1(T)=\frac{D_1(T)}{D_0^{\circ}(T)}=\left(1-qT\right)\left(1-q^2T\right)\prod_{i=3}^{\infty}\left(1-q^iT\right).
	\end{align*}	
	Theorem \ref{th215} shows that $D(T)$ can be expressed as a product of $D^{\circ}_i(T)$ as follow.
	\begin{align*}
		D(T)=\prod^{n+2}_{i=1}\prod^{N(i)}_{j=1}D^{\circ}_i(T)
		=(1-T)(1-qT)^{n+3}\cdots,
	\end{align*}
	Note that $\LF^*(f,T)^{(-1)^{n+1}}$ is a polynomial of degree $2(n+1)$ if $f$ is non-degenerate.
	Combining formula (\ref{eq24}), we obtain
	\begin{align*}
		\LF^*(f,T)^{(-1)^{n+1}}
		=\frac{D(T)D(q^2T)^{\binom{n+2}{2}}\cdots}{D(qT)^{n+2}D(q^3T)^{\binom{n+2}{3}}\cdots} =(1-T)(1-qT)\prod^{2n}_{i=1}(1-\beta_iT),
	\end{align*}
where $|\beta_i| = q^{\frac{w_i}{2}}\leq q^{\frac{n+2}{2}}$. That is, $w_i \leq n+2$. 

If $\beta_i$ is a reciprocal root of $\LF^*(f,T)^{(-1)^{n+1}}$, the conjugate $\overline{\beta_i}$ is a reciprocal root of the 
conjugate L-function 
$$\overline{\LF^*(f,T)}^{(-1)^{n+1}} = {\LF^*(-f,T)}^{(-1)^{n+1}}.$$
By Theorem \ref{thas}, the Newton polygon and Hodge polygon coincide at the end points. Applying this to the 
product $\LF^*(f,T)^{(-1)^{n+1}}\overline{\LF^*(f,T)}^{(-1)^{n+1}}$,  we deduce 
that 
$$ 2(n+1)^2 =2(\sum_{i=1}^n2i+n+1)= \ord_q (1\cdot q^2 \cdot \prod_{i=1}^{2n} \beta_i \overline{\beta_i}) = 2 + \sum_{i=1}^{2n} w_i\leq 2+2n(n+2)=2(n+1)^2.$$
It follows that the inequality must be an equality, that is, all $w_i = n+2$.  
\end{proof}

Formula (\ref{eq34}) relates $\LF^*(f,T)$ to $\LF_{n}(b,T)$. The valuations for the reciprocal roots and poles of $\LF_{n}(b,T)$ follow from Theorem \ref{thm34} and \ref{thm35}.
\begin{Th}\label{thm37}
	Suppose $p\nmid (n+1)$. We have
	\begin{align*}
		\LF_{n}(b,T)^{(-1)^{n+1}}=(1-T)^{(n+1)}
		\prod_{j=2}^{n}\left(1-q^{j-1}T\right)^{\binom{n}{j}(-1)^{j-1}}	
		\prod^{2n}_{i=1}(1-\alpha_iT).
	\end{align*}
	For each $1\leq i\leq 2n$, the reciprocal root $\alpha_i$ satisfies $|\alpha_i|= q^{\frac{n}{2}}$. If $p\equiv 1\bmod (n+1)$, the slope sequence of the $\alpha_i$'s  is given by $\{0, 1, 1, 2, 2, \ldots, n-1, n-1, n\}$.
\end{Th}

Based on weights of toric L-function, we get the following slightly more precise upper bound for its associated exponential sum.
\begin{Cor}
	If $p\nmid (n+1)$, we have
		\begin{align*}
		|S_{k,n}(b)+\frac{(q^k-1)^n-(-1)^n(q^k+1)}{q^k}|\leq 2nq^{\frac{nk}{2}}.
	\end{align*}
\end{Cor}
\begin{proof}
	Theorem \ref{thm35} implies that 
	\begin{align*}
		|S_k^*(f)-(-1)^n(q^k+1)|\leq 2nq^{\frac{(n+2)k}{2}}.
	\end{align*}
Combining formula (\ref{eq32}), we get the bound for $S_{k,n}(b)$.
\end{proof}

\begin{remark}
	We finish this paper with two open problems on the estimates of inverted Kloosterman sums. If $n+1$ is divisible by $p$, the related Laurent polynomial $f$ is degenerate and thus the results for toric exponential sums are not tenable. In this case, it is an open problem to determine the optimal square root cancellation for $S_n(\chi,b)$ in general. The case $n=1$ with $p=2$ is already handled in \cite{Kat1995}. 
The second question concerns the $q$-adic slope sequence. If $p$ is not equivalent to $1$ modulo $n+1$, the Newton polygon corresponding to $f$ is strictly above its Hodge polygon. Under this assumption, can one still obtain the explicit $q$-adic slope sequence?
\end{remark}


\nocite{*}
\bibliographystyle{amsalpha}
\bibliography{ref}


\end{document}